\documentclass[11pt]{article}
\usepackage{amssymb}
\usepackage{amsmath}
\usepackage{graphicx}
\usepackage{acronym}
\usepackage{amsfonts}
\usepackage{color}
\usepackage{cite}
\usepackage[margin=2.8in]{geometry}%
\usepackage{float}

\usepackage{tikz}
\usepackage{mathpazo}
\usetikzlibrary{arrows,calc}

\newcommand{\qed}{\bull \medskip}

\newtheorem{theorem}{Theorem} %[section]
\newtheorem{proposition}{Proposition} %[section]
\newtheorem{lemma}{Lemma}%[section]
\newtheorem{definition}{Definition}
\newtheorem{example}{Example}%[section]
\def\be{\begin{example}}
\def\ee{\end{example}}
\def\bt{\begin{theorem}}
\def\et{\end{theorem}\bigskip}
\def\bl{\begin{Lemma}}
\def\el{\end{Lemma}\bigskip}
\def\ep{\end{Proposition}\bigskip}
\def\bp{\begin{Proposition}}
\def\bd{\begin{definition}}
\def\ed{\end{definition}}

\newcommand{\alglist}{
\begin{list}{Step 1}
{\setlength{\leftmargin}{1.1 in}\setlength{\labelwidth}{1.0 in}} }

\def\qed{\hfill {$ \Box $} \medskip}

\begin{document}
\title{\bf A Lower Bound for the Cardinality of Function Basis of Tensor Invariants
\footnote{This work is partially supported by National Science Foundation of China (Grant No. 11771328), and the Hong Kong Research Grant Council (Grant Nos. PolyU 15302114, 15300715, 15301716 and 15300717).}}
 \author{Shenglong Hu\footnote{Department of Mathematics, School of Science, Hangzhou Dianzi University, Hangzhou 310018, China. \textit{On leave from} School of Mathematics, Tianjin University, Tianjin 300350, China. Email: timhu@tju.edu.cn.}\quad  and \quad Liqun Qi\footnote{Department of Applied Mathematics, The Hong Kong Polytechnic University, Hung Hom, Kowloon, Hong Kong. Email: liqun.qi@polyu.edu.hk.}}

\date{\today}
\maketitle

{\large

\begin{abstract}
In this article, we give a proof for that the cardinality of a function basis of the invariants for a finite dimensional real vector space by a compact group is lower bounded by the intuitive difference of the dimensions of the vector space and the group. An application is given to the space of third order three dimensional symmetric and traceless tensors, showing that each minimal integrity basis is an irreducible function basis, which solves a problem in applied mechanics.
\medskip

\end{abstract}

\medskip

%%%%%%%%%%%%%%%%%%%%%%%
\section{Introduction}\label{sec:intro}
Tensors play important role in applied mechanics, invariants of tensors are essential tools for studying physical objects, and bases of invariants of a given tensor space are of particular interests and foundation for the sake of simplification in turn \cite{z}. In the literature, invariants of tensors in applied mechanics are mainly categorized into two classes: polynomial (algebraic) invariants and functional invariants. Polynomial invariants refer to invariants of the tensor space which are polynomials in terms of the coordinates in a basis of the tensor space. Likewise, functional invariants are more general, referring to invariants which are functions of the coordinates.

Polynomial invariants of a tensor space have been studied over a century, dating back to prominent mathematicians such as Gordan, Hilbert, Weyl, etc. Extensive classical books on these can be found \cite{w,o,h,s93}. We refer to Section~\ref{sec:tensor} for a brief introduction on polynomial invariants of tensors. For several tensor spaces, (minimal) integrity bases are known for polynomial invariants \cite{sb,oka}, and algorithms existing for computing an integrity basis for the ring of polynomial invariants of a given tensor space \cite{s93}. While, the cardinality of a minimal integrity basis may consist of hundreds of invariants \cite{oka,bp}. It is thus still difficult to analyze or compute from such a huge number of invariants, and makes the simplificative purpose less convincing. Moreover, the cardinality of a minimal integrity basis is the same for a tensor space \cite{s}. Hence, people turn to the more general function basis, hoping that an irreducible function basis will have a cardinality with significantly a smaller number of invariants.

However, in general function bases and irreducible function bases of a given tensor space are much difficult to determine \cite{z,oka,cqzz}. A good news is that in some cases, an irreducible function basis can be derived from an integrity basis. Nevertheless, there still lacks a systematic way to determine an irreducible function basis from a (minimal) integrity basis. Before solving the question of determining irreducible function bases, a basic question needs an answer first: what is the cardinality of an irreducible function basis of a tensor space? It turns out that this is a question as hard as determining the irreducible function basis. Thus, instead one ask the following question: \textit{is there any lower or upper bounds for the cardinality of an irreducible function basis of a tensor space}?

In the article, we are trying to answer partially this basic question and giving a lower bound for the cardinality of a function basis.

The rest article is organized as follows. Preliminaries on terminologies and tensor invariants will be presented in Section~\ref{sec:preliminary}. Section~\ref{sec:function} is for proving the main theorem, showing that the cardinality of a function basis is lowered bounded by the the dimension of the tensor space minus the dimension of the acting group, under mild assumptions.
Section~\ref{sec:app} will give an application of the lower bound result, showing that each minimal integrity basis is an irreducible function basis.

%%%%%%%%%%%%%%%%%%%%%%%
\section{Preliminaries}\label{sec:preliminary}
%%%%%%%%%%%%%%%%%%%%%%%%%%%%
In this section, we will present necessary notions and results from tensor invariant theory and quotient manifolds for the preparation of the main theory in Section~\ref{sec:function}.
%%%%%%%%%%%%%%%%%%%
\subsection{Tensor Invariants}\label{sec:tensor}
Let $m>1$ and $n>1$ be given integers. The space of real tensors $\mathcal A$ of order $m$ and dimension $n$ is formed by all tensors (a.k.a. hypermatrices) with entries $a_{i_1\dots i_m}\in\mathbb R$, the field of real numbers, for all $i_j\in\{1,\dots,n\}$ and $j\in\{1,\dots,m\}$. It is denoted as $\operatorname{T}(m,n)$. Let $\operatorname{Gl}(n,\mathbb R)\subset\mathbb R^{n\times n}$ be the \textit{general linear group} of real matrices. Let $G\subseteq \operatorname{Gl}(n,\mathbb R)$ be a subgroup. We then have a natural group representation $G\rightarrow \operatorname{Gl}(\operatorname{T}(m,n),\mathbb R)$, the real general linear group of the linear space $\operatorname{T}(m,n)$, via
\[
(g\cdot\mathcal T)_{j_1\dots j_m}:=\sum_{i_1}^n\dots\sum_{i_m=1}^ng_{j_1i_1}\dots g_{j_mi_m}t_{i_1\dots i_m}.
\]
A linear subspace $V$ of $\operatorname{T}(m,n)$ is \textit{$G$-stable} if $g\cdot v\in V$ for all $g\in G$ and $v\in V$.

Of particular interests in this article are the compact subgroups $\operatorname{O}(n,\mathbb R)$ (the orthogonal group) and $\operatorname{SO}(n,\mathbb R)$ (the special orthogonal group), both of which are Lie groups \cite{v}.

In $\operatorname{T}(m,n)$, the subspace of symmetric tensors $\operatorname{S}(m,n)$ is $\operatorname{Gl}(n,\mathbb R)$-stable, and thus $G$-stable for every subgroup $G$. Likewise, inside $\operatorname{S}(m,n)$, the subspace of symmetric and traceless tensors $\operatorname{St}(m,n)$ is $\operatorname{O}(n,\mathbb R)$-stable, thus $\operatorname{Ol}(n,\mathbb R)$-stable. Recall that a symmetric tensor $\mathcal T\in \operatorname{S}(m,n)$ is traceless if
\[
\sum_{i=1}^nt_{iii_3\dots i_m}=0\ \text{for all }i_3,\dots,i_m\in\{1,\dots,n\}.
\]
A well-known fact is that the dimension of $\operatorname{S}(m,n)$ as a linear space is ${n+m-1\choose n-1}$, and that of $\operatorname{St}(m,n)$ is ${n+m-1\choose n-1}-{n+m-3\choose n-1}$.

Associated to a linear subspace $V\subseteq\operatorname{T}(m,n)$ is an algebra $\mathbb R[V]$, generated by the dual basis of $V$. Once a basis of $V$ is fixed, an element $f\in \mathbb R[V]$ can be viewed as a polynomial in terms of the coefficients of $v\in V$ in that basis. Let $G\subseteq \operatorname{Gl}(n,\mathbb R)$ be a subgroup and $V$ be $G$-stable. Then, we can induce a group action of $G$ on $\mathbb R[V]$ via
\[
(g\cdot f)(v)=f(g^{-1}\cdot v)\ \text{for all }g\in G\ \text{and }v\in V.
\]
With this group action, some elements of $\mathbb R[V]$ are fixed points for the whole $G$, i.e.,
\[
g\cdot f=f\ \text{for all }g\in G,
\]
which form a subring $\mathbb R[V]^G$ of $\mathbb R[V]$ \cite{w,o}. Elements of $\mathbb R[V]^G$ are \textit{invaraints} of $V$ under the action of $G$. It is well-known that $\mathbb R[V]^G$ is finitely generated. A generator set is called an \textit{integrity basis}.  In an integrity basis, if none of the generators is a polynomial of the others, it is a \textit{minimal integrity basis}. Given a subspace $V$ and group $G$, minimal integrity bases may not be unique, but their cardinalities are the same as well as the lists of degrees of the generators \cite{s}. Invariants in $\mathbb R[V]^G$ are polynomials, always referred as \textit{algebraic invaraints}.

Likewise, one can consider \textit{function invariants} \cite{o}. A function $f\colon V\rightarrow\mathbb R$ is an invariant if
\[
f(v)=f(g\cdot v)\ \text{for all }g\in G.
\]
The set of function invariants of $V$ is denoted as $\mathcal I(V)$.
If there is a set of generators such that each function invariant can be expressed as a function of the generators, it is called a \textit{function basis}. Similarly, if none of the generators is a function of the others in a function basis, it is called an \textit{irreducible function basis}.

%%%%%%%%%%%%%%%%%%%%%%
\subsection{Quotient Manifold by Lie Groups}\label{sec:quotient}
A real vector space $V$ of finite dimension has a natural manifold structure. Any given equivalence relation $\sim$ on $V$ defines a quotient structure with elements being the \textit{equivalence classes}
\[
V/\sim:=\{[v]\mid v\in V\}\ \text{with }[v]:=\{u\in V\mid v\sim u\}.
\]
The set $V/\sim$ is the \textit{quotient} of $V$ by $\sim$, and $V$ is the \textit{total space} of $V/\sim$.
The quotient $V/\sim$ is a \textit{quotient manifold} if the natural projection $\pi : V\rightarrow V/\sim$ is a submersion.
$V/\sim$ admits at most one manifold structure making it being a quotient manifold \cite[Proposition~3.4.1]{ams}. It may happen that $V/\sim$ has a manifold structure but fails to be a quotient manifold. Whenever $V/\sim$ is indeed a quotient manifold, we call the equivalence relation $\sim$ \textit{regular}.

Let $G$ be any compact Lie group and $V$ a finite dimensional real linear space. Suppose that $V$ is a representation of $G$, i.e., there is a group homomorphism $G\rightarrow\operatorname{Gl}(V,\mathbb R)$. Then, there is a natural equivalence relation given by $G$ as
\[
v\sim u\ \text{if and only if }g\cdot v=u\ \text{for some }g\in G\ \text{and } \text{for all }v, u\in V.
\]
The quotient under this equivalence is sometimes denoted as $V/G$, which is the set of orbits of the group action of $G$ on $V$. Suppose in the following that the group action is continuous.
Then, with the compactness of $G$, it can be shown that $V/G$ is a quotient smooth manifold, since the graph set
\[
\{(v,u)\mid [v]=[u]\}\subset V\times V
\]
is closed \cite[Proposition~3.4.2]{ams}.

Note that the fibre of the natural projection $\pi$ is the equivalence class $\pi^{-1}(\pi(v))=[v]$ for each $v\in V$. If $[v]$ is not a discrete set of points for some $v\in V$, then the dimension of $V/\sim$ is strictly smaller than the dimension of $V$ \cite[Proposition~3.4.4]{ams}.

In the following, we consider subspaces of  the linear space of tensors of order $m$ and dimension $n$, i.e., $V\subseteq \operatorname{T}(m,n)$.
%%%%%%%%%%%%%
\begin{lemma}\label{lem:dimension}
Let $V\subseteq\operatorname{T}(m,n)$ be a linear space containing $\operatorname{St}(m,n)$ and $G=\operatorname{O}(n,\mathbb R)$ or $\operatorname{SO(n,\mathbb R)}$. Then, we have $\operatorname{dim}(V/G)<\operatorname{dim}(V)$, and
\begin{equation}\label{eq:dim}
\operatorname{dim}(V/G)\geq \operatorname{V}-\operatorname{dim}(G).
\end{equation}
\end{lemma}
\begin{proof}
By \cite[Proposition~3.4.4]{ams}, if there is one point $v\in V$ such that $[v]$ is not a set of discrete points, then $\operatorname{dim}(V/G)<\operatorname{dim}(V)$, and $\operatorname{dim}(V/G)=\operatorname{V}-\operatorname{dim}([v])$, where $[v]$ is regarded as an embedded submanifold of $V$.

Note that $[v]$ is the orbit of $G$ acting on the element $v$. Thus, the dimension of $[v]$ cannot exceed the dimension of $G$. Consequently, the dimension bound \eqref{eq:dim} follows if we can find a point $v\in V$ such that $[v]$ is not a discrete set of points.

First of all, we show that $[v]$ cannot be a discrete set of points for the group $G=\operatorname{SO(n,\mathbb R)}$ for some $v\in V$.

It is easy to see that the stabilizers $G_v=G$ cannot hold through out $v\in V$. Thus, there exists an orbit $[v]$ with more than one element. Suppose that $[v]$ is a discrete set of more than two points.
For any given two discrete points $v_1,v_2\in [v]$, there exist $g_1,g_2\in G$ such that
\[
v_i=g_i\cdot v\ \text{for all }i=1,2
\]
by the definition of $[v]$. Since $\operatorname{SO(n,\mathbb R)}$ is a connected manifold, there is a smooth curve $g(t)$ starting from $g(0)=g_1$ ending at $g(1)=g_2$. By the definition,
\[
g(t)\cdot v\in [v]\ \text{for all }t\in [0,1].
\]
Since the group action is smooth, we see that $v_1$ and $v_2$ is thus connected, contradicting the discreteness.

Since $\operatorname{SO(n,\mathbb R)}$ is one half connected component of  $\operatorname{O(n,\mathbb R)}$, the result for  $\operatorname{O(n,\mathbb R)}$ follows immediately.
\qed
\end{proof}

Note that both $\operatorname{O(n,\mathbb C)}$ and
$\operatorname{SO(n,\mathbb C)}$ are algebraic varieties.

%%%%%%%%%%%
\section{Cardinality of Function Basis}\label{sec:function}

The next result is \cite[Theorem~11.112]{v}, see also the classical book \cite{w}.
%%%%%%%%%%%%%%
\begin{lemma}[Separability]\label{lem:sep}
Let $G$ be a compact group and $V$ a real vector space representing $G$. Then the orbits of $G$ acting on $V$ are separated by the invariants $\mathbb R\mathbb [V]^G$.
\end{lemma}
The conclusion may fail in the complex case.

The concepts of function invariants and functional independence of invariants can be found in classical textbooks, see for example \cite[Page 73]{o}.

The analysis for integrity and minimal integrity bases of $V$ for some $G$ is more sophisticated and approachable than function basis. Nevertheless, an exciting fact that an integrity basis is also a function basis holds in most interesting cases.
We will present this result in Theorem~\ref{thm:function}.
%%%%%%%%%%%%%%%%%%%%%%%%%%%%%%%%%%%%%%%%%%%%%
\begin{theorem}[Function Basis]\label{thm:function}
Let $G$ be a compact group and $V$ a finite dimensional real linear vector space representing $G$. Then, any integrity basis of $\mathbb R[V]^G$ is a function basis.
\end{theorem}

\begin{proof}It is well-known that the ring of polynomial invariants $\mathbb R[V]^{G}$ is finitely generated, whose minimal set of generators is an integrity basis.

The orbits of $G$ on $V$ are separable, i.e., $p(u)=p(v)$ for all $p\in \mathbb R[V]^{G}$ if and only if $u=g\cdot v$ for some $g\in G$ by Lemma~\ref{lem:sep}.  Let $\mathcal P:=\{p_1,\dots,p_r\}$ be an integrity basis. We have a map
\[
\mathbb P: V\rightarrow \mathbb P(V)\ \text{with }v\mapsto (p_1(v),\dots,p_r(v))^\mathsf{T},
\]
where $\mathbb P(V)$ is the image of $\mathbb P$ on $V$.
Actually, this map is defined over $V/G$, as each $p_i\in \mathcal P$ is an invariant. Moreover, this map, with $G/V\rightarrow\mathbb P(V)$, is onto and one to one, following from the separability of $\mathbb R[V]^{G}$ on $V$ and the fact that each algebraic invariant is generated by $p_1,\dots,p_r$. Thus, there is an inverse map
\[
\mathbb P^{-1}: \mathbb P(V)\rightarrow G/V.
\]
In summary,
we can conclude that $[v]$ (the equivalent class in $G/V$) for any $v\in V$ can be determined by the values of $p_1(v),\dots,p_r(v)$.
On the other side, each invariant in $\mathcal I(V)$, the set of invariants of $V$, is a function over $V/G$. Thus, we have a chain of functions
\[
V\rightarrow\mathbb P(V)\leftrightarrow V/G\rightarrow \mathbb R.
\]
Read throughout the above chain, we get that the integrity basis $\mathcal P$ gives a function basis for $\mathcal I(V)$.
\qed
\end{proof}

When conditions in Theorem~\ref{thm:function} are fulfilled, we can derive a function basis and even an irreducible function basis from an integrity basis or minimal integrity basis. A function basis derived from an integrity basis is called a \textit{polynomial function basis}, and an irreducible function basis derived from a minimal integrity basis is called an \textit{irreducible polynomial function basis}. In the following, we will give a lower bound for the cardinality of a polynomial function basis.

Since $\mathbb R[V]^G$ is finitely generated \cite[Theorem~11.114]{v} and has no nilpotent elements, it follows from \cite[Theorem~1.3]{s77} that
that $V/G$ is a (quotient) variety. It is the variety determined by the coordinate ring $\mathbb R[V]/(\mathbb R[V]^G)$.

%As $V$ is an irreducible linear space, so is $V/G$.
%%%%%%%%%%%%%%%%%%%%%%
%\section{The Cardinality Theorem}\label{sec:cardinality}
%%%%%%%%%%%%%
\begin{theorem}[The Cardinality Theorem]\label{thm:card}
Let $G$ be a compact group of dimension $d$ and $V$ a finite dimensional real linear vector space representing $G$ of dimension $n>d$. Then, any polynomial function basis has cardinality being not smaller than $n-d$.
\end{theorem}

\begin{proof}
Let $\{p_1,\dots,p_r\}\subset P[V]^{G}$ be a polynomial function basis. We must have that for each pair $u,v\in V$
\[
p_i(u)=p_i(v)\ \text{for all }i\in\{1,\dots,r\}
\]
will implies
\[
[u]=[v],
\]
since each polynomial in $P[V]^{G}$ is a function of $p_1,\dots,p_r$, and $P[V]^G$ separates the orbits of $V/G$ \cite{w}.

We therefore have that the mapping
\[
\mathcal P : V/G\rightarrow \mathbb R^r
\]
given by
\[
\mathcal P([v])=(p_1(v),\dots,p_r(v))^\mathsf{T}
\]
is a one to one regular map. Obviously, we can consider the mapping
\[
\mathcal P : V/G\rightarrow \overline{\mathcal P(V/G)}\subseteq\mathbb R^r
\]
whenever $\mathcal P$ is not dominant. Now, the map
\[
\mathcal P : V/G\rightarrow \overline{\mathcal P(V/G)}
\]
is a dominant morphism. Then, if $r<n-d\leq \operatorname{dim}(V/G)$, each fibre of $\mathcal P^{-1}(\mathbf y)$ for $\mathbf y\in \mathcal P(V/G)$ will have dimension at least $\operatorname{dim}(V/G)-\operatorname{dim}(\overline{\mathcal P(V/G)})\geq n-d-r\geq 1$ \cite[Proposition~6.3]{b}. This contradicts the separability of the set $\{p_1,\dots,p_r\}$ on the orbits of $V/G$ immediately.
\qed
\end{proof}

%%%%%%%%%%%%%%%%%%%%%%%%%%
\section{An Application to Applied Mechanics}\label{sec:app}
In this section, we give an application of the lower bound theorem for function basis in applied mechanics.

Third order three dimensional symmetric and traceless tensors are of fundamental importance in physics such as liquid crystal, etc. In the literature, Smith and Bao\cite{sb} derived minimal integrity bases for the space of third order three dimensional symmetric and traceless tensors. Their minimal integrity bases have four polynomial invariants, of degrees two, four, six, and eight respectively. While, for a long time, it is unclear whether these minimal integrity bases are also irreducible function bases or not. In very recent, Chen, Qi and Zou \cite{cqzz} applied an intelligence method proving that the Smith-Bao minimal integrity basis is indeed an irreducible function basis. In the following, we will strengthen this result, and show that each minimal integrity basis of the space $\operatorname{St}(3,3)$ is an irreducible function basis.
%----------------------------
\begin{proposition}\label{prop:33}
Each minimal integrity basis of $\operatorname{St}(3,3)$ is an irreducible function basis.
\end{proposition}

\begin{proof}
First note that the dimension of $\operatorname{St}(3,3)$ is $7$. Thus, the dimension of $\operatorname{St}(3,3)/\mathbb O(3)$ is at least $4$.  It follows from Theorem~\ref{thm:card} that an irreducible function basis will have cardinality at least $4$.

On the other hand, every minimal integrity basis of $\operatorname{St}(3,3)$ will have the same cardinality $4$ \cite{s}, which is of course an upper bound for the cardinality of irreducible function bases derived from them.

As the lower bound is equal to the upper bound for the cardinality of the irreducible function basis, the conclusion follows.
\qed
\end{proof}

We give a final remark to conclude this article. The lower bound given in Theorem~\ref{thm:card} can be strict. Recently, Chen et. al. \cite{clqzz} gave an irreducible function basis of eleven invariants for the space of third order three dimensional symmetric tensors, for which the lower bound given by Theorem~\ref{thm:card} is seven. While, it is still possible to find an irreducible function basis with the cardinality being the lower bound, which shall be termed a \textit{minimal irreducible function basis}.

%%%%%%%%%%%%%%%%%%%%%%%

\end{document}